\documentclass[11pt,a4paper]{article}

\usepackage{amssymb}
\usepackage{amsmath, amsthm}
\usepackage{epsfig}
\usepackage{setspace}

\usepackage{geometry}

\def\RR{{\bf R}}

\numberwithin{equation}{section}

\newtheorem{Thm}{Theorem} 
\newtheorem{Prop}[Thm]{Proposition}

\theoremstyle{definition}

\title{On uncrossing games for skew-supermodular functions}
\author{Hiroshi HIRAI \\
Department of Mathematical Informatics, \\
Graduate School of Information Science and Technology,   \\
University of Tokyo, Tokyo, 113-8656, Japan.\\
\texttt{\normalsize hirai@mist.i.u-tokyo.ac.jp}\\
}

\begin{document}
\maketitle

\begin{abstract}
	In this note, 
	we consider the uncrossing game for a skew-supermodular function $f$, 
	which is a two-player game with players, Red and Blue, and abstracts 
	the uncrossing procedure in the cut-covering linear program associated with $f$.
	Extending the earlier results by Karzanov 
	for $\{0,1\}$-valued skew-supermodular functions, 
	we present an improved polynomial time strategy for Red to win,  
	and give a strongly polynomial time uncrossing procedure 
	for dual solutions of the cut-covering LP as its consequence.
	We also mention its implication on the optimality of laminar solutions.
\end{abstract}
\begin{center}
\end{center}

\section{Introduction}
Let $V$ be a finite set, and let ${\cal S}(V)$ denote the set of all bi-partitions of $V$. 
A member $\{X,V \setminus X\}$ of ${\cal S}(V)$ 
is also denoted by $X$ or $V \setminus X$ (if no confusion occurs).
A pair $X,Y$ of members in ${\cal S}(V)$ is said be {\em crossing} if 
$X \cap Y$, $V \setminus (X \cup Y)$, $X \setminus Y$, and $Y \setminus X$ are all nonempty.
A family ${\cal F} \subseteq {\cal S}(V)$  is called {\em laminar}  if 
there is no crossing pair in ${\cal F}$.
A function $f:{\cal S}(V) \to \RR_+$ is called {\em skew-supermodular} 
if it satisfies
\[
f(X) + f(Y) \leq \max\{f(X \cap Y) + f(X \cup Y), f(X \setminus Y) + f(Y \setminus X)  \}
\]
for every pair $X,Y \in {\cal S}(V)$.
For a skew-supermodular function $f$, 
we consider the following game involving two players {\em Blue} and {\em Red}, 
which we call the {\em uncrossing game}.
\begin{description}
	\item[Uncrossing game for skew-supermodular function $f$]
	\item[Input:] A family ${\cal F} \subseteq {\cal S}(V)$.
	\item[Step 1:] If ${\cal F}$ is laminar, then Red wins; the game terminates.
	\item[Step 2:] Otherwise, Red chooses a crossing pair $(X,Y)$ in ${\cal F}$,  
	chooses $(X',Y') \in \{ (X \cap Y, X \cup Y), (X \setminus Y, Y \setminus X)\}$ 
	satisfying 
	\[
	f(X) + f(Y) \leq f(X') + f(Y'),
	\]
	and replaces $X, Y$ by $X',Y'$ in ${\cal F}$.
	\item[Step 3:] Blue returns one of $X,Y$ to ${\cal F}$. Go to {\bf step 1}.
\end{description}
Here we assume that Red has an evaluation oracle of $f$.

The uncrossing game abstracts the uncrossing procedure in combinatorial optimization, 
and was originally introduced by Hurkens, Lov\'asz, Schrijver and Tardos~\cite{HLST87}. 
The original formulation is the following. 
For an input family ${\cal F}$ of subsets of $V$, 
if ${\cal F}$ is a chain, then Red wins.
Otherwise Red chooses an incomparable pair $(X,Y)$ in ${\cal F}$ 
and replaces $(X,Y)$ by $(X \cap Y, X \cup Y)$ in ${\cal F}$.
Then Blue returns one of $X$ and $Y$ to ${\cal F}$.
Hurkens, Lov\'asz, Schrijver and Tardos showed that there is a strategy 
(i.e., a way of choosing $(X,Y)$) for Red to win 
after a polynomial number $O(|V||{\cal F}|)$ of iterations. 
Karzanov~\cite{Kar96} considered a symmetric generalization on cross-closed families.
Here a family ${\cal S} \subseteq {\cal S}(V)$ is called {\em cross-closed} 
if for $X,Y \in {\cal S}$,  
$X \cap Y, X \cup Y$ belong to ${\cal S}$ or $X \setminus Y, Y \setminus X$ belong to ${\cal S}$.
In his uncrossing game, 
the input is a family ${\cal F} \subseteq {\cal S}$, and  
Red chooses a crossing pair $(X,Y)$ in ${\cal F}$ and 
chooses $(X',Y') \in \{ (X \cap Y, X \cup Y), (X \setminus Y, Y \setminus X)\}$ 
with $X',Y' \in {\cal S}$.
As was noticed by Karzanov, 
the original uncrossing game can be viewed as a special case of 
the cross-closed family on $V \cup \{s,t\}$ consisting of $X$ 
with $|X \cup \{s,t\}| = 1$.
He showed that there is a strategy for Red to win 
after $O(|V|^4|{\cal F}|)$ iterations. 

A cross-closed family ${\cal S}$ 
is naturally identified with a $\{0,1\}$-valued skew-supermodular function $f$
defined by $f(X) := 1$ for $X \in {\cal S}$ and $f(X) := 0$ otherwise.
By this identification, the uncrossing game on the cross-closed family ${\cal S}$ 
reduces to our setting. 
We do not know whether the converse reduction is possible.
Also Karzanov's strategy~\cite{Kar96} seems not to be applied to our generalization. 
Indeed, his strategy includes selection rules 
of type ``if $X \not \in {\cal S}$, then Red takes..."\cite[p. 222]{Kar96}.
The main aim of this note is to present an improved Red-win strategy for our generalization. 
\begin{Thm}\label{thm:main}
	For every skew-supermodular function $f$ and every input ${\cal F}$,
	there exists a strategy for Red to win after $O(|V|^3|{\cal F}|)$ iterations 
\end{Thm}
The rest of the paper is organized as follows.
In Section~\ref{sec:proof}, we prove Theorem~\ref{thm:main}.
Our strategy basically follows Karzanov's one in high level, 
and incorporates nontrivial modifications in an essential part.
In Section~\ref{sec:implications}, 
we explain implications of this result on  
skew-supermodular cut-covering linear programs, 
which constitute an important subclass of linear programs 
appearing from a wide variety of network design problems in combinatorial optimization. 
As was noted by \cite{Kar96}  (for $\{0,1\}$-skew-supermodular cases), 
the uncrossing game naturally arises from 
the uncrossing procedure of dual solutions, 
and a Red-win strategy gives rise to a strongly polynomial time uncrossing algorithm.
Our strategy in Theorem~\ref{thm:main} 
is applicable to general skew-supermodular cut-covering linear programs 
beyond that treated in~\cite{Kar96}.
We also present an unexpected consequence on 
the optimality property of laminar dual solutions.

\section{Proof}\label{sec:proof}
Let ${\cal F}$ be a set of bi-partitions on $V$.
A member $X \in {\cal F}$ is said to be {\em trivial} (in ${\cal F}$)
if no member of ${\cal F}$ is crossing with $X$.
If $Z$ is not crossing with $X$ and $Y$, 
then it is not crossing with $X \cap Y$, $X \cup Y$, $X \setminus Y$, and $Y \setminus X$.
In particular, a trivial member remains trivial after uncrossings.
Therefore we can assume that 
a trivial member is removed from ${\cal F}$ whenever it appears.
Define an equivalence relation $\sim_{\cal F}$ on $V$ 
by $i \sim_{\cal F} j$ if there is no $X \in {\cal F}$ 
with $|\{i,j\} \cap X| = 1$.
An equivalence class of this relation is called an {\em atom}.
If $i \not \sim_{\cal F} j$, then $i$ and $j$ are said to be {\em separated}.
Obviously each $X \in {\cal F}$ is a disjoint union of atoms.
Therefore we can identify the ground set $V$ with the set of all atoms.
Notice that $i \sim_{\{X \cap Y, X \cup Y\}} j$ implies $i \sim_{\{X,Y\}} j$, and 
$i \sim_{\{X \setminus Y, Y \setminus X\}} j$ implies $i \sim_{\{X,Y\}} j$.
Therefore, 
if ${\cal F}'$ is obtained from ${\cal F}$ 
through iterations of the game, then $\sim_{{\cal F}'}$ coarsens  $\sim_{{\cal F}}$, 
i.e., $i \sim_{\cal F} j$ implies $i \sim_{{\cal F}'} j$.
By this fact together with the removal of trivial members, 
it happens that several atoms in ${\cal F}$ are joined to a single atom in ${\cal F}'$.
In this case, the cardinality of the ground set decreases.

\paragraph{1.} 
We first consider the essential case of input ${\cal F}$ to
which the argument in \cite{Kar96} is not directly applicable.
Let $V = \{1,2,\ldots,n\}$ $(n \geq 4)$. 
Suppose that ${\cal F}$ satisfies:
\begin{itemize}
	\item[(A)] 
	${\cal F}$ is the disjoint union of two subsets ${\cal A}$ and ${\cal B}$ 
	such that every member in ${\cal A}$ takes form $\{1,2,\ldots, i\}$ 
	for some $2 \leq i \leq n-2$, and every member of ${\cal B}$ takes form $\{2,3,\ldots, j\}$ 
	for some $3 \leq j \leq n-1$.
\end{itemize}
In the following, we denote $\{i,i+1,i+2,\ldots,j \}$ by $[i,j]$.
We may assume that ${\cal B} \neq \emptyset$.
Let $d (\geq 3)$ be the minimum number 
for which $[2,d]$ belongs to ${\cal B}$.
We can assume that $[1,i]$ belongs to ${\cal A}$ for $i=2,3,\ldots,d-1$; 
otherwise some $i$ and $i+1$ are not separated, and are joined to a single element.
We give a strategy for Red 
to keep the game of form (A) 
and to decrease $n + |{\cal B}|$ after $O(d)$ iterations.

Initially, Red evaluates $f([1,d-1]) + f([2,d])$, $f(\{1\}) + f(\{d\})$, and $f([2,d-1]) + f([1,d])$.  
Suppose that $f([1,d-1]) + f([2,d]) \leq f(\{1\}) + f(\{d\})$ holds.
Red chooses $X = [1,d-1], Y = [2,d]$, chooses $X' = X \setminus Y = \{1\}$, 
and $Y' = Y \setminus X = \{d\}$, and
replaces $X,Y$ by $X',Y'$; both $X',Y'$ are singletons (trivial) and vanish.
If Blue returns $X = [1,d-1]$, then $|{\cal B}|$ decreases.
If Blue returns $Y = [2,d]$, then $|{\cal B}|$ does not change, 
$d$ and $d-1$ are not separated, and hence $n$ decreases.

Suppose that $f([1,d-1]) + f([2,d]) > f(\{1\}) + f(\{d\})$ holds. 
By skew-supermodularity, $f([1,d-1]) + f([2,d]) \leq f([2,d-1]) + f([1,d])$ must hold.
Suppose that $d=3$. 
Red chooses $X = [1,2]$, $Y = [2,3]$,  $X' = X \cup Y = [1,3]$, and $Y' = X \cap Y = \{2\}$.
Then $\{2\}$ vanishes; in particular $|{\cal B}|$ does not increase.
If Blue returns $X = [1,2]$, 
then $|{\cal B}|$ decreases.
If Blue returns $Y = [2,3]$, then $2$ and $3$ are not separated, and $n$ decreases.

Suppose that $d > 3$. 
Red computes the smallest $k \in [2, d -1]$ such that 
\begin{equation}
f([1,l]) + f([2,l+1]) \leq   
f([2,l]) + f([1,l+1]) \quad (l = k,k+1,\ldots, d-1). \label{eqn:2}
\end{equation}
Such an index $k$ actually exists since $f([1,d-1]) + f([2,d]) \leq f([2,d-1]) + f([1,d])$. 
If $k > 2$, then it holds $f([1,k-1]) + f([2, k]) >  f([2, k-1]) + f([1,k])$, 
and by skew-supermodularity, it holds
\begin{equation}\label{eqn:1}
f([1,k-1]) + f([2,k]) \leq f(\{1\}) + f(\{k\}). 
\end{equation}
Also, in the case of $k=2$, (\ref{eqn:1}) holds 
since $f(\{1\}) + f(\{2\}) \leq f(\{1\}) + f(\{2\})$.
Adding inequalities (\ref{eqn:2}), we obtain
\begin{equation*}
f([1,k]) + f([2,d]) \leq f([1,d]) + f([2,k]).
\end{equation*}
Red chooses $X = [1,k]$, $Y = [2,d]$, $X' = X \cup Y = [1,d]$, and $Y' = X \cap Y = [2,k]$, and
replaces $X,Y$ by $X',Y'$.
If Blue returns $X = [1,k]$, 
then $|{\cal B}|$ does not change, $d$ decreases, and Red goes to the initial stage above.
Suppose that Blue returns $Y = [2,d]$.
If $k=2$, then $Y'$ is a singleton and vanishes, and $2$ and $3$ are not separated; 
hence $n + |{\cal B}|$ decreases.
Suppose that $k > 2$.
In the next iteration, $|{\cal B}|$ increases by one, and $d$ becomes $k$.
By (\ref{eqn:1}),
Red chooses $X = [1,k-1]$, $Y = [2,k]$, 
$X' = X \setminus Y = \{1\}$, and $Y' = Y \setminus X = \{k\}$.
Both $X'$ and $Y'$ are singletons and vanish in the next iteration.
If Blue returns $[1,k-1]$, 
then ${\cal B}$ decreases by one, 
$k$ and $k+1$ are not separated (since now $[1,k]$ does not exist), 
and $n + |{\cal B}|$ is smaller than that in two iterations before.
If Blue returns $[2,k]$, 
then $k-1$ and $k$ are not separated, 
$n + |{\cal B}|$ is equal to that in two iterations before,  $d = k$ is smaller than before, and
Red goes to the initial stage above.

Summarizing, by using the above strategy, 
Red can keep the game of form (A) 
and decrease $n + |{\cal B}|$ after $O(d)$ iterations.
Thus Red wins after $O(n^2)$ iterations.
We remark that the above strategy is 
easily adapted for the case where Blue is allowed to return none of $X,Y$.

\paragraph{2.} The rest of arguments is exactly the same as that given in \cite{Kar96}, and is sketched as follows.
A subset $X \subseteq V$ is said to be {\em $2$-partitioned} for ${\cal F}$ 
if $X$ intersects at most two atoms with respect to ${\cal F}$.
Suppose that 
\begin{itemize}
	\item[(B)] 
the input ${\cal F}$ is the disjoint union of two laminar families ${\cal C}$ and ${\cal D}$
such that for each $X \in {\cal C}$, $X$ or $V \setminus X$ is $2$-partitioned for ${\cal D}$.
\end{itemize}
In this case, for any $X \in {\cal C}$ the family ${\cal D} \cup \{X\}$ satisfies the condition (A) 
after removing trivial members in ${\cal D} \cup \{X\}$.
Indeed, ${\cal D} \cup \{X\}$ consists of $X$ 
and members $Y_1,Y_2,\ldots,Y_m$ in ${\cal D}$ crossing with $X$.
We can assume that $X$ is $2$-partitioned for ${\cal D}$.
By this condition, $X$ is the disjoint union of $Z_1$ and $Z_2$ 
such that $X \cap Y_{i} \in \{Z_1, Z_2\}$ for $i=1,2,\ldots,m$.
We can assume that both $Z_1$ and $Z_2$ are nonempty, 
and $Z_2 = X \cap Y_i$ for each $i$ (by $Y_i \leftrightarrow V \setminus Y_i$). 
By the laminarity of ${\cal D}$ together with $\emptyset \neq Z_1 \subseteq V \setminus (Y_i \cup Y_j)$ and $\emptyset \neq Z_2 \subseteq Y_i \cap Y_j$, we have $Y_i \subset Y_j$ or $Y_j \subset Y_i$ for $i \neq j$. 
By rearranging them, 
we have $Z_2 \subset Y_1 \subset Y_2 \subset \cdots \subset Y_m$.  
Atoms are $Z_1$, $Z_2$, $Y_1 \setminus Z_2$, $V \setminus (Y_{m} \cup Z_1)$, and
$Y_{i+1}\setminus Y_{i}$ for $i=1,2,\ldots,m-1$. 
Thus the situation reduces to (A) in the setting of 
${\cal A} = \{ [1,2]\}$, ${\cal B} = \{ [2,j] \mid 3 \leq j \leq m+2\}$, and $n = m+3$.

The strategy for Red is the following.
Red chooses {\em maximal} $\{X, V \setminus X\} \in {\cal C}$ in the sense that
$X$ is $2$-partitioned for ${\cal D}$
and there is no $\{Y, V \setminus Y\} \in {\cal C}$ such that 
$X \subset Y$ (proper inclusion) and $Y$ is 2-partitioned for ${\cal D}$. 
As above, the family ${\cal D} \cup \{X\}$ satisfies the condition (A) 
(after removing trivial members).
Therefore Red plays the game within ${\cal D} \cup \{X\}$ 
and obtains a laminar family ${\cal D}'$ by the above strategy after $O(|V|^2)$ iterations.
Since ${\cal D}'$ is obtained from ${\cal D} \cup \{X\}$ by uncrossings,  $\sim_{{\cal D}'}$ coarsens $\sim_{{\cal D} \cup \{X\}}$ 
(see the beginning of Section~\ref{sec:proof}).
By this fact together with the maximality of $X$ and the laminarity of ${\cal C}$, the union of
${\cal C}' := {\cal C} \setminus \{X\}$ and ${\cal D}'$ satisfies the condition~(B). 
Let ${\cal C} \leftarrow {\cal C}'$ and  ${\cal D} \leftarrow {\cal D}'$. 
Red repeats the same procedure for ${\cal C}$ and ${\cal D}$, and wins after $|{\cal C}|$ steps.
The total number of iterations is $O(|{\cal C}| |V|^2) = O(|V|^3)$; 
recall that the size of any laminar family on $V$ is $O(|V|)$.

Suppose finally that ${\cal F}$ is arbitrary.
Red chooses a (maximal) laminar subset ${\cal C}$ of ${\cal F}$, 
and chooses an arbitrary $X \in {\cal B} := {\cal F} \setminus {\cal C}$.
Then the union of ${\cal C}$ and ${\cal D} = \{X\}$ obviously satisfies~(B).
Red applies the above strategy for ${\cal C} \cup {\cal D}$, and obtains a laminar family ${\cal C}'$ 
after $O(|V|^3)$ iterations.
Let ${\cal C} \leftarrow {\cal C}'$ and ${\cal B} \leftarrow {\cal B} \setminus \{X\}$.
Red repeats the same procedure for ${\cal C}$ and ${\cal B}$ (until ${\cal B} = \emptyset$), 
and wins after $O(|{\cal F}|)$ steps.
The total number of iterations is $O(|V|^3|{\cal F}|)$, as required.

\section{Implications}\label{sec:implications}
As mentioned in Introduction, 
the uncrossing game abstracts 
the uncrossing procedure arising from a class of cut-covering linear programs.
Let $G = (V,E)$ be an undirected graph with edge-cost $a:E \to \RR_+$, 
and let $f: {\cal S}(V) \to \RR_+$ be a skew-supermodular function.
The {\em skew-supermodular cut-covering LP} is the following linear program:
Minimize the cost $\sum_{e \in E} a(e) x(e)$ over all edge-weights $x:E \to \RR_+$ satisfying
the covering constraint:
\[
\sum_{e \in \delta X} x(e) \geq f(X) \quad (X \in {\cal S}(V)),
\]
where $\delta X$ denotes the set of edges $e = ij \in E$ with $i \in X$ and $j \not \in X$.
This class of LP and its variation capture a wide variety of network design problems 
and their fractional relaxations.
Examples include matching, T-join, network synthesis, survivable network, 
traveling salesman, Steiner tree/forest, connectivity augmentation and so on; see, e.g., \cite{Jain01,Schrijver}.
An important feature of this LP is that its dual always admits a {\em laminar} optimal solution.
The dual LP is given as:
Maximize $\sum_{X \in {\cal S}(V)} \lambda(X) f(X)$ over all $\lambda: {\cal S}(V) \to \RR_+$ satisfying
\[
\sum_{X \in {\cal S}(V): e \in \delta X} \lambda(X) \leq a(e) \quad (e \in E).
\] 
A feasible solution $\lambda$ is called {\em laminar} 
if its nonzero support 
${\cal F}(\lambda) := \{ X \in {\cal S}(V) \mid \lambda(X) \neq 0\}$ is laminar.
Then there always exists a laminar optimal solution.
This useful property has played key roles in algorithm design and analysis:
Edmonds' blossom algorithm~\cite{Edmonds65} for weighted matching works with a laminar dual solution, 
which enables us to avoid keeping exponential number of inequalities/variables for matching polytope.
Jain's iterative rounding algorithm~\cite{Jain01} for survivable network was obtained
by analyzing the larminarity property of skew-supermodular covering LPs.
The existence of a laminar dual solution can be seen from the following standard uncrossing argument.
Let $\lambda$ be an arbitrary feasible solution.
Suppose that ${\cal F}(\lambda)$ is not laminar.
Choose a crossing pair $(X,Y)$ in ${\cal F}(\lambda)$, and choose $(X',Y') \in 
\{(X \cap Y, X \cup Y), (X \setminus Y, Y \setminus X) \}$ 
with $f(X) + f(Y) \leq f(X') + f(Y')$.
Decrease $\lambda$ by $\alpha := \min \{\lambda(X), \lambda(Y)\}$ on $X$ and $Y$, 
and increase $\lambda$ by $\alpha$ on $X'$ and $Y'$.
The objective value does not decrease (by skew-supermodularity) and
the feasibility is also preserved.
This operation is called an {\em uncrossing}.
For rational $\lambda$, 
we obtain a laminar solution after a finite number of uncrossings, 
where $\lambda$ is said to be {\em uncrossed}.

The uncrossing process 
gives rise to the uncrossing game with input ${\cal F} = {\cal F}(\lambda)$, 
as mentioned in \cite{Kar96} (under the setting of a cross-closed family).
Red chooses a crossing pair $X,Y$ in ${\cal F}$, and replaces $X,Y$ by $X',Y'$ in ${\cal F}$.
Blue returns one $\tilde X$ of $X,Y$ 
for $\lambda(\tilde X) > \min \{\lambda(X), \lambda(Y)\}$; 
Blue returns none of them if $\alpha = \lambda(X) = \lambda(Y)$.
Then ${\cal F} = {\cal F}(\lambda)$ holds in the next iteration.
Suppose that $\lambda$ is integer-valued. 
Observe that $\sum_{X \in {\cal S}(V)} |X||V \setminus X| \lambda(X)$ strictly decreases in one uncrossing.
Therefore, in any choices of uncrossing pairs, 
this process terminates after $O( |V|^2 |{\cal F}(\lambda)| \| \lambda\|)$ iterations,  
where $\|\lambda\| := \max_{X \in {\cal S}(V)} \lambda(X)$.
By using the Red-win strategy in Theorem~\ref{thm:main}, 
we can conduct the uncrossing procedure in time polynomial in $|V|$ and $|{\cal F}(\lambda)|$, 
not depending on the bit length of $\lambda$. Thus we have:
\begin{Thm}\label{thm:uncrossing}
	Any $\lambda:{\cal S}(V) \to \RR_+$ can be uncrossed 
	in time polynomial of $|V|$ and~$|{\cal F}(\lambda)|$.
\end{Thm}
This is a natural extension of \cite[Theorem 2]{Kar96} to general skew-supermodular functions.
In~\cite[Section 3.3]{HP}, we found that
the Red-win strategy of the uncrossing game also brings 
an interesting optimality property of laminar solutions, 
where we proved this property 
for $\{0,1\}$-valued skew-supermodular functions
(as a consequence of Karzanov's uncrossing algorithm).
Now we can state and prove this for general skew-supermodular functions.
\begin{Prop}
	Let $\lambda$ be a nonoptimal laminar feasible function.
	For any $\epsilon > 0$,
	there exists a laminar feasible solution $\lambda^*$ such that 
	$\|\lambda - \lambda^*\| \leq \epsilon$ and 
	the objective value of $\lambda^*$ is greater than that of $\lambda$.
\end{Prop}
\begin{proof}
	The following proof method is due to \cite{HP}.
	We can assume that $\lambda \neq 0$.
	We can choose a positive integer $N$ such that every $\lambda:{\cal S}(V) \to \RR_+$ 
	can be uncrossed by at most $N$ uncrossings 
	in the strategy of Theorem~\ref{thm:main}.
	Choose $\epsilon' > 0$.
	Since $\lambda$ is not optimal, 
	there exists a feasible (not necessarily laminar) solution $\lambda'$ 
	such that $\|\lambda - \lambda'\| \leq \epsilon'$ and the objective value of $\lambda'$ is greater than that of $\lambda$; by convexity, 
	$\lambda'$ can be taken from the segment between $\lambda$ and an optimal solution.
	Apply the uncrossing procedure to $\lambda'$ according to the strategy of Theorem~\ref{thm:main}.
	Then we obtain a laminar feasible solution $\lambda^*$ 
	with the objective value not less than that of~$\lambda'$.

    We show that $\|\lambda - \lambda^*\| \leq \epsilon$ if $\epsilon'$ is sufficiently small. 
    Let $\lambda^k$ denote $\lambda'$ after $k$ uncrossings; $\lambda^0 = \lambda'$.
	Choose $\epsilon'$ with $\epsilon' \leq 2^{- N} \min_{X \in {\cal F}(\lambda)} \lambda(X)$.  
	Then it holds
	\begin{equation}\label{eqn:ast}
	\min_{Z \in {\cal F}(\lambda)} \lambda^k(Z) 
	\geq \min_{Z \in {\cal F}(\lambda)} \lambda(Z) - 2^k \epsilon' \geq 2^{k} \epsilon' \geq \max_{Z \not \in {\cal F}(\lambda)} \lambda^k(Z) \quad (k=0,1,2,\ldots,N-1).
	\end{equation}
	We show (\ref{eqn:ast}) by induction on $k$. 
	In the case of $k=0$, this indeed holds (by $\|\lambda - \lambda'\| \leq \epsilon'$). 
	The second inequality is obvious from the definition of $\epsilon'$.
	Suppose that (\ref{eqn:ast}) holds for $k < N-1$.
		Each uncrossing step chooses $X,Y$ 
		so that at least one of $X,Y$ does not belong to laminar family ${\cal F}(\lambda)$.
	By (\ref{eqn:ast}), in the $(k+1)$-th uncrossing, 
	$\alpha (= \min \{\lambda(X), \lambda(Y)\})$ 
	is attained at ${\cal S}(V) \setminus {\cal F}(\lambda)$ and is bound by $2^k \epsilon'$. 
	Thus $\min_{Z \in {\cal F}(\lambda)} \lambda^{k+1}(Z) \geq  \min_{Z \in {\cal F}(\lambda)} \lambda^{k}(Z) - 2^{k} \epsilon' \geq 2^{k+1} \epsilon'$, and 
	$\max_{Z \not \in {\cal F}(\lambda)} \lambda^{k+1}(Z) \leq \max_{Z \not \in {\cal F}(\lambda)} \lambda^{k}(Z) + 2^{k} \epsilon' \leq 2^{k+1} \epsilon'$.
	Thus (\ref{eqn:ast}) holds after the $(k+1)$-th uncrossing.
	
	In particular, $\alpha$ is attained at ${\cal S}(V) \setminus {\cal F}(\lambda)$ 
	in every step $k \leq N$, and $\|\lambda^{k+1} - \lambda^k\| \leq 2^{k} \epsilon'$.
	Hence we have $\|\lambda^* - \lambda'\| \leq \sum_{k=1}^{N} \|\lambda^k - \lambda^{k-1}\| \leq \sum_{k=1}^N 2^{k-1} \epsilon' = (2^N -1) \epsilon'$, and 
	$\|\lambda - \lambda^*\| \leq \|\lambda - \lambda'\| + \|\lambda^* - \lambda'\| \leq 2^N \epsilon'$.  Thus, by choosing $\epsilon'$ as 
	$\epsilon' \leq 2^{- N} \min \{\epsilon, \min_{X \in {\cal F}(\lambda)} \lambda(X)\}$,
	we obtain the desired result.
\end{proof}
From the view of convexity,
this local-to-global optimality property is obvious without the laminarity requirement.
The set of all laminar feasible solutions is not convex 
in the space of all functions on ${\cal S}(V)$.
Nevertheless this proposition says that 
the objective behaves convex or unimodal
over the space of laminar feasible solutions.
This may suggest a laminarity-preserving primal-dual algorithm,  like Edmonds' blossom algorithm,
for general skew-supermodular cut-covering LPs.

\section*{Acknowledgments}
We thank the referees for helpful comments.
The work was partially supported by JSPS KAKENHI Grant Numbers 25280004, 26330023, 26280004.


\begin{thebibliography}{1}
\small
%
%
\bibitem{Edmonds65}
J. Edmonds, 
Maximum matching and a polyhedron with 0,1-vertices,
{\em Journal of Research of the National Bureau of Standards} {\bf 69B} (1965), 125--130. 

\bibitem{HP}
H. Hirai and G. Pap,
Tree metrics and edge-disjoint S-paths, 
{\em Mathematical Programming, Series A} {\bf 147} (2014), 81--123.


\bibitem{HLST87}
C. A. J. Hurkens, L. Lov\'asz, A. Schrijver, and \'E. Tardos,
How to tidy up your set-system?
in:{\em Combinatorics} 
(Proceedings Seventh Hungarian Colloquium on Combinatorics, Eger, 1987;
A. Hajnal. L. Lov\'asz, V.T.S\'os, eds.), 
North-Holland, Amsterdam, 1988, pp. 309--314.

\bibitem{Jain01}
K. Jain, 
A factor 2 approximation algorithm for the generalized Steiner network problem, 
{\em Combinatorica} {\bf 21} (2001), 39--60.

\bibitem{Kar96}
A. V. Karzanov,
How to tidy up a symmetric set-system by use of uncrossing operations,
{\em Theoretical Computer Science} {\bf 157} (1996), 215--225.
%
\bibitem{Schrijver}
A. Schrijver, {\em Combinatorial Optimization},
Springer-Verlag, Berlin, 2003.
%



\end{thebibliography}
\end{document}